\documentclass{amsart}   	
\usepackage{parskip}
\usepackage{amssymb,amsthm}
\usepackage{amsthm}
\usepackage{graphicx}
\usepackage{enumitem}
\usepackage{tikz}
\usepackage{tikz-cd}
\usetikzlibrary{calc}
\usetikzlibrary{tikzmark} 
\usetikzlibrary{decorations.pathmorphing}
\usepackage{hyperref}
\usepackage{adjustbox}
\usepackage{mathtools}
\usepackage{mathabx}
\usepackage[all]{xy}

\tikzset{curve/.style={settings={#1},to path={(\tikztostart)
    .. controls ($(\tikztostart)!\pv{pos}!(\tikztotarget)!\pv{height}!270:(\tikztotarget)$)
    and ($(\tikztostart)!1-\pv{pos}!(\tikztotarget)!\pv{height}!270:(\tikztotarget)$)
    .. (\tikztotarget)\tikztonodes}},
    settings/.code={\tikzset{quiver/.cd,#1}
        \def\pv##1{\pgfkeysvalueof{/tikz/quiver/##1}}},
    quiver/.cd,pos/.initial=0.35,height/.initial=0}

\tikzset{tail reversed/.code={\pgfsetarrowsstart{tikzcd to}}}
\tikzset{2tail/.code={\pgfsetarrowsstart{Implies[reversed]}}}
\tikzset{2tail reversed/.code={\pgfsetarrowsstart{Implies}}}
\tikzset{no body/.style={/tikz/dash pattern=on 0 off 1mm}}

\graphicspath{ {./images/} }

\newtheorem{thm}{Theorem}[section]

\newtheorem{lem}[thm]{Lemma}
\newtheorem{corollary}[thm]{Corollary}

\theoremstyle{definition}
\newtheorem{defn}[thm]{Definition}

\newtheorem{ex}[thm]{Example} 
 
\newtheorem{conjecture}[thm]{Conjecture}

\newcommand{\uxa}{\ensuremath{(\underline{X},\underline{A})}}

\newcommand{\ux}{\ensuremath{(\underline{X},\underline{\ast})}}

\newcommand{\caa}{\ensuremath{(\underline{CA},\underline{A})}}

\newcommand{\cyy}{\ensuremath{(\underline{CY},\underline{Y})}}

\usepackage[utf8]{inputenc}

\title{Homotopy exponents of polyhedral products}
\author[B. Eldridge]{Briony Eldridge}
\address{
Beijing Key Laboratory of Topological Statistics and Applications for Complex Systems, Beijing Institute of Mathematical Sciences and Applications (BIMSA)}
\email{eldridge@bimsa.cn}
\date{May 2025}

\begin{document}
\begin{abstract}
     We study Moore's conjecture and homotopy exponents for polyhedral products. For $\caa^K$ where each $A_i$ is finite and has torsion-free homology, we prove that if $\caa^K$ is rationally hyperbolic, then it has no homotopy exponent at any odd prime. Under the additional hypothesis that $\Sigma A_i$ is homotopy equivalent to a finite-type wedge of simply-connected spheres, we show Moore's conjecture holds for $\caa^K$. We also give criteria ensuring that polyhedral products $\ux^K$ are rationally hyperbolic, mod-$p^r$-hyperbolic for every $r\geq1$ at all but finitely many primes, and have no homotopy exponent at those primes. When each $X_i$ is a nonempty finite product of simply-connected spheres and $K$ is not a simplex, these conclusions hold at every prime.
\end{abstract}
\maketitle
\section{Introduction}

Moore's Conjecture proposes a profound link between the rational homotopy groups and the torsion components of the homotopy groups in finite $CW$-complexes. It suggests that the behaviour of one reflects constraints on the other. Let $X$ be a finite $CW$-complex. The space $X$ is \emph{rationally elliptic} if it has finitely many rational homotopy groups and \emph{rationally hyperbolic} if it has infinitely many rational homotopy groups. The rational dichotomy theorem of F{\'e}lix, Halperin and Thomas \cite[Chapter 33]{FelixHalperinThomas2001} states that if $X$ is rationally hyperbolic, then its rational homotopy groups grow exponentially; there is no polynomial growth case for the rational homotopy groups of a rationally hyperbolic space.  

The torsion condition appearing in Moore's Conjecture is expressed in terms of homotopy exponents. The \emph{homotopy exponent} of $X$ at a prime $p$ is the least power of $p$ that annihilates the $p$-torsion in the homotopy groups of $X$. If no such power exists, then $X$ has $p$-torsion of arbitrarily large order, and we say that $X$ has no exponent at $p$. Moore's conjecture states the following. 
\begin{conjecture}
(\textbf{Moore's Conjecture})
    Let $X$ be a finite, simply-connected CW-complex. Then the following are
equivalent:
\begin{itemize}
    \item  $X$ is rationally elliptic;
    \item $X$ has a finite homotopy exponent at every prime $p$;
    \item $X$ has a finite homotopy exponent at some prime $p$.
\end{itemize}
\end{conjecture}

We can reformulate Moore's conjecture to be a statement about rationally hyperbolic spaces; a finite, simply-connected $CW$-complex is rationally hyperbolic if and only if it has no finite homotopy exponent at any prime.  We can also consider a torsion version of hyperbolicity, where a $CW$-complex $X$ is $\mathrm{mod}$-$p^r$ \emph{hyperbolic} if the number of direct $\mathbb{Z}/p^r$ summands in the homotopy groups grows exponentially, and in \cite{HuangTheriault2024}, the authors conjectured that if $X$ is rationally hyperbolic, then it is $\mathrm{mod}$-$p^r$ hyperbolic for all primes $p$, $r \geq 1$. 

Moore's conjecture has been verified in many cases. Spheres \cite{Jamessphere, Toda1956}, finite $H$-spaces \cite{Long1978}, $H$-spaces with finitely generated cohomology \cite{ChacholskiPitschSchererStanley2008} and odd primary Moore spaces \cite{Neisendorfer} are all examples of rationally elliptic spaces with finite exponents at all primes. On the other side of the conjecture, wedges of simply-connected spheres, most torsion-free two-cell complexes \cite{NeisendorferSelick1982}, and torsion-free suspensions \cite{Selick_1983} are examples of rationally hyperbolic spaces that have no exponent at any prime.

In this work, we study Moore's conjecture, homotopy exponents, and mod-$p^r$ torsion for polyhedral products. Polyhedral products are spaces obtained by gluing products of spaces together according to the underlying simplicial complex. Their homotopy type reflects both the individual homotopy-theoretic properties of the constituent spaces and the combinatorics of the associated complex. More precisely, let $K$ be a simplicial complex on the vertex set $[m] = \{1,\dots,m\}$, and let $\uxa =~ \{(X_i,A_i)\}_{i=1}^m$ be a sequence of pointed $CW$-pairs. Define 
\[\uxa^{\sigma} = \prod_{i=1}^m Y_i, \;\;\: \; Y_i = \begin{cases}
    X_i & i \in \sigma \\
    A_i & i \notin \sigma.
\end{cases}\]

The polyhedral product is defined as \[\uxa^K = \bigcup_{\sigma \in K} \uxa^{\sigma} \subseteq \prod_{i =1}^m X_i.\]
In recent years, substantial progress has been made in understanding these spaces from multiple perspectives, including toric topology, combinatorics, geometric group theory, and complex geometry \cite{bahri2024polyhedralproductsfeatureshomotopy}. 
An important family consists of polyhedral
products on the pairs $(D^2_i,S^1_i)$. These are called \emph{moment-angle complexes}. In \cite{HAO_2018}, the authors showed that Moore's conjecture holds for generalised moment-angle complexes (polyhedral products on the pair $(D^{n}_i,S^{n-1}_i)$), and gave necessary and sufficient conditions for a polyhedral product to be rationally elliptic or hyperbolic. In \cite{Kim2018}, it was shown that real moment-angle complexes (polyhedral products on the pair $(D^{1}_i,S^{0}_i)$) also satisfy Moore's conjecture. Both real and generalised moment-angle complexes are examples of polyhedral products of the form $\caa^K$. More generally, in \cite{HAO_2018}, it was shown that polyhedral products associated to arbitrary pairs $\uxa$ are rationally elliptic only in highly restrictive cases. Our first result concerns homotopy exponents for rationally hyperbolic $\caa^K$. 

\begin{thm}[Theorem \ref{thm: mooretorsionfreepp} in text]
    Let $K$ be a simplicial complex on the vertex set $[m]$ and let $\caa$ be a sequence of pairs such that for all $i \in [m]$, $A_i$ is a finite path-connected $CW$-complex, and each $A_i$ has torsion-free homology. Then if $\caa^K$ is rationally hyperbolic, $\caa^K$ has no homotopy exponent for any odd prime $p$.
\end{thm} 

When $\Sigma A_i$ is homotopy equivalent to a finite-type wedge of simply-connected spheres and each $A_i$ is rationally nontrivial, Theorem \ref{thm: aiinw} gives a stronger conclusion:  Moore's conjecture holds for $\caa^K$. For polyhedral products of the form $\caa^K$, Theorem \ref{thm: mooretorsionfreepp} significantly restricts where a potential counterexample to Moore's conjecture arising from polyhedral products could occur. In the rationally hyperbolic case, such a counterexample must either occur at the prime $2$ or involve spaces $A_i$ with torsion in homology.

In \cite{stanton2026anicksconjecturepolyhedralproducts}, the authors showed that there exists some finite set of primes such that, after localising away from these primes, Moore's conjecture holds for $\caa^K$. In general, this set of primes cannot be made explicit. Theorem \ref{thm: mooretorsionfreepp} verifies Moore's conjecture for rationally hyperbolic $\caa^K$ at odd primes, which reduces the number of primes needed to localise away from dramatically. 

We then consider polyhedral products of the form $\ux^K$ and results which hold after localisation at all but finitely many primes. 
\begin{thm}[Theorem \ref{thm: uxmoores} in text]
Let $\ux^{K}$ be a polyhedral product. Suppose that there is a minimal missing face $\sigma$ of $K$ such that, for each $j \in \sigma$, the space $X_j$ is a simply connected finite $CW$-complex which is rationally elliptic and rationally non-trivial. After localising at all but finitely many primes, there exists a wedge $W$ of infinitely many simply connected $p$-local spheres such that $\Omega W$ retracts off
\[
\Omega\bigl((\ux^K)_{(p)}\bigr).
\]
Consequently, $\ux^K$ is mod-$p^r$-hyperbolic for every $r\geq 1$, is rationally hyperbolic, and has no homotopy exponent at $p$. In particular, Moore's conjecture holds for $\ux^K$ at $p$.
\end{thm}

If $X_1,\dots,X_m$ are simply connected finite $CW$-complexes which are rationally elliptic and rationally non-trivial, then $\ux^K$ satisfies Moore's conjecture for all but finitely many primes. Furthermore, when each $X_i$ is a nonempty product of simply connected spheres, Corollary \ref{5.4} strengthens this conclusion to every prime. Thus the combinatorics of the minimal missing faces of $K$, together with homotopy decompositions of spaces $X_i$ at large primes provide a direct way of detecting rational and torsion hyperbolicity in polyhedral products.

This paper is structured as follows. Section 2 recalls general results concerning Moore's conjecture, homotopy exponents, and mod-$p^r$ hyperbolicity. Section 3 recalls general results on polyhedral products and known results concerning Moore's conjecture for polyhedral products. In Section 4, we prove new results concerning Moore's conjecture and homotopy exponents for polyhedral products of the form $\caa^K$. Finally, Section 5 considers polyhedral products of the form $\ux^K$ and proves results concerning Moore's conjecture, rational hyperbolicity, and mod-$p^r$ hyperbolicity which hold for all but finitely many primes.

The author would like to thank both Stephen Theriault and Lewis Stanton for many helpful suggestions that improved this work. The author would also like to thank the referee for insightful comments that have improved this work.

\section{General results on Moore's conjecture}

We begin by recalling Moore's conjecture, and state some classical results that are used throughout the rest of this paper. 
\begin{defn}
   A simply-connected finite $CW$-complex $X$ is \emph{rationally elliptic} if it has finitely many rational homotopy groups. Otherwise it is \emph{rationally hyperbolic}.
\end{defn}
\begin{ex}\label{2.2}
The sphere $S^n$ is rationally elliptic for $n \geq 2$ \cite{Serre}, and the wedge $S^m \vee S^n$ is rationally hyperbolic for $m,n \geq 2$ \cite{Hilton}. 
\end{ex}

\begin{defn}
    Let $p$ be a prime and $r \geq 1$. A simply-connected finite $CW$-complex $X$ is \emph{mod}-$p^r$ \emph{hyperbolic} if the number of direct $\mathbb{Z} / p^r$-summands in $\oplus^m_{i=1}\pi_i(X)$ grows exponentially as $m$ increases. 
\end{defn}

\begin{ex}\label{2.4}
    In \cite{Boyde2022}, it was shown that $S^m \vee S^n$ is mod-$p^r$ hyperbolic for all primes $p$ and all $r \geq 1$. It remains unknown if $S^n$ is mod-$p^r$ hyperbolic or not for any prime $p$ or $r \geq 1$. 
\end{ex}
\begin{defn}
    Let $p$ be a prime and $r \geq 1$. A simply-connected finite $CW$-complex $X$ has \emph{homotopy exponent} $p^r$ if $r$ is the least power of $p$ that annihilates the $p$-torsion in $\pi_*(X)$, denoted $\mathrm{exp}_p(X)=r$. 
\end{defn}

\begin{ex}\label{2.6}
    In \cite{moore1}, \cite{Moore2}, it was shown that for $p \geq 3$, $\mathrm{exp}_p(S^{2n+1}) = n$. When $p =2$, \cite{MR721456} together with \cite{Jamessphere} gives an upper bound  for the $2$-exponent for odd-dimensional spheres as $\mathrm{exp}_2(S^{2n+1}) \leq~\frac{3n}{2} + \epsilon$, where $\epsilon= 0$ if $n$ is even, and $\epsilon = \frac{1}{2}$ if $n$ is odd. Furthermore, \cite[Theorem 1.1]{DAVIS198939} gives an implicit lower bound: for every $r \geq 1$, there exists $N(r)$ such that $\mathrm{exp}_2(S^{n}) \geq r$ whenever $n \geq N(r)$, and so the $2$-primary exponents of spheres are unbounded as their dimensions increase. In \cite{NeisendorferSelick1982}, it was shown that if $n,m \geq 2$, $S^n \vee S^m$ has no exponent at any prime.
\end{ex}

We now state Moore's conjecture.
\begin{conjecture}
(\textbf{Moore's Conjecture})
    Let $X$ be a finite, simply-connected CW-complex. Then the following are
equivalent:
\begin{itemize}
    \item  $X$ is rationally elliptic;
    \item $X$ has a finite homotopy exponent at every prime $p$;
    \item $X$ has a finite homotopy exponent at some prime $p$.
\end{itemize}
\end{conjecture}

 Moore's conjecture asserts that a space $X$ is rationally hyperbolic if and only if $X$ has no exponent at any prime $p$. Moore's conjecture is known to hold for spheres, wedges of spheres and mod-$p^r$ Moore spaces with $p > 2$ or $r >1$, however, the conjecture is far from resolved. Here we prove a result that is known in the literature, and include a proof for completeness. Let $\mathcal{P}$ be the collection of $H$-spaces which are homotopy equivalent to a finite-type product of spheres and loops on simply-connected spheres.
\begin{lem}\label{lem: PWspaces-moore}
    If $X$ is a finite simply-connected $CW$-complex such that $\Omega X \in \mathcal{P}$, then Moore's conjecture holds for $X$.
\end{lem}
\begin{proof}
      We first establish some notation. Let $Y= S^n$ for $n \in \{1,3,7\}$ or $Y = \Omega S^m$ for $m \notin \{2,4,8\}$. Then any element of $\mathcal{P}$ can be written as a finite-type product $\prod_{i \in I} Y_i$ for some indexing set $I$, where each $Y_i$ is an instance of $Y$. Let $\Omega X \in \mathcal{P}$. If $X$ is rationally elliptic, then as each $Y$ has nontrivial rational homotopy, only finitely many $Y_i$ can appear in the product. Each $Y_i$ appearing in this product has a finite $p$-exponent at every prime $p$ (when $p = 2$, this is due to \cite[Corollary 1.13]{Jamessphere}, and for odd primes, this is due to \cite[Theorem 8.10]{Toda1956}). Therefore $X$ has a finite homotopy exponent at every prime $p$. 
      
      Now assume $X$ is rationally hyperbolic. Then there must be infinitely many $Y_i$ in the product decomposition of $\Omega X$. Because the product is of finite-type, the dimension of each sphere appearing in the product decomposition of $\Omega X$ must be unbounded. Fix a prime $p$. For odd primes $p$, the even-dimensional case is covered by the standard $p$-local splitting $\Omega S^{2n} \simeq_{(p)} S^{2n-1} \times \Omega S^{4n-1}$. By \cite[Corollary 1.3]{Moore2}, for odd-dimensional spheres, $\mathrm{exp}_p(S^{2n+1}) = n$. As the dimensions of the spheres appearing in the product decomposition of $\Omega X$ are unbounded, this implies that the $p$-primary homotopy exponent of $X$ is unbounded.
      Now consider $p=2$. Fix $r \geq 1$. By \cite[Theorem 1.1]{DAVIS198939}, there exists a stable element of order $2^r$ in the image of the $2$-primary stable $J$-homomorphism whose sphere of origin is bounded above by some integer $N(r)$. Hence, for every $m \geq N(r)$, some homotopy group of $S^m$ contains an element of order $2^r$. Since the dimensions of the spheres appearing in the product decomposition of $\Omega X$ are unbounded, we may choose a factor such that its homotopy group contains an element of order $2^r$. As $r$ is arbitrary, the $2$-primary homotopy exponent is unbounded. 
\end{proof}

Moore's conjecture relates rational hyperbolicity to the absence of finite homotopy exponents. However, it does not address the growth rate of the number of direct $\mathbb{Z}/p^r$-summands in the homotopy groups. The following was conjectured in \cite{HuangTheriault2024}. 

\begin{conjecture}\cite[Conjecture 1.6]{HuangTheriault2024}
    Let $X$ be a simply-connected finite $CW$-complex. If $X$ is rationally hyperbolic then it is mod-$p^r$ hyperbolic for all primes $p$ and $r \geq 1$. 
\end{conjecture}

We can also consider a partial version of Moore's conjecture. For a finite simply-connected $CW$-complex $X$, we say Moore's conjecture \emph{holds at a prime }$p$ if for that fixed prime $p$, $X$ has no finite $p$-homotopy exponent if and only if $X$ is rationally hyperbolic. The following two classical results are examples of this prime-local form of the conjecture.

\begin{thm}\cite[Corollary 12]{Selick_1983}\label{thm: mooretorsionfree}
    Let $p$ be an odd prime. Let $X$ have the homotopy type of a finite $CW$-complex such that $H_*(X)$ is torsion-free. If $\Sigma X$ is rationally hyperbolic, then $\Sigma X$ has no homotopy exponent at $p$. \qed
\end{thm}
Theorem \ref{thm: mooretorsionfree} verifies Moore's conjecture for rationally hyperbolic torsion-free suspensions at odd primes. 

We conclude this section by giving a $p$ local homotopy equivalence for the loop space of a rationally elliptic space. 

\begin{thm}\label{thm : loop decomp}\cite[Theorem 1]{McW}
    Let $X$ be a simply connected finite $CW$-complex that is rationally elliptic. Then for all but finitely many primes $p$, there is a homotopy equivalence 
    \[\Omega X_{(p)} \simeq \prod_{i} S^{2m_i-1}_{(p)} \times \prod_{j} \Omega S^{2n_j-1}_{(p)}. \eqno \qed\]
\end{thm}

\section{General results on polyhedral products}

This section covers preliminary results on polyhedral products, and known results concerning polyhedral products and Moore's conjecture. 

 Let $K$ be a simplicial complex on $[m]$, and let $L$ be a proper subcomplex of $K$. Assume that $L$ is on the vertex set $[l]$ for $l < m$. We can view $L$ as a simplicial complex on the vertex set $[m]$ by considering $l+1 , \dots, m$ as \emph{ghost vertices}. The simplicial inclusion $L \to K$ induces a map of polyhedral products 
\[\uxa^L \times \prod_{i=l+1}^m A_i \to \uxa^K.\]

If $L$ is a full subcomplex with no ghost vertices, we obtain a left inverse for the induced map $\uxa^L \to \uxa^K$.  

\begin{lem}\label{lem: left inverse}\cite[Lemma 2.2.3]{Denham_2007}
    Let $K$ be a simplicial complex on $[m]$, and let $L$ be a full subcomplex of $K$. Then the induced map of polyhedral products $\uxa^L \to \uxa^K$ has a left inverse \[r: \uxa^K \to \uxa^L. \eqno \qed\]
\end{lem}

\begin{thm}\cite[Theorem 2.1, Theorem 2.3]{HAO_2018}\label{fibreandloop}
    Let $K$ be a simplicial complex on the vertex set $[m]$ and let $\uxa$ be any sequence of pointed, path-connected $CW$-pairs. Then there is a homotopy fibration 
    \[\cyy^K \to \uxa^K \to \prod_{i=1}^m X_i\] where $Y_i$ is the homotopy fibre of the inclusion $A_i \to X_i$. Furthermore, this homotopy fibration splits after looping to give 
    \[\Omega \uxa^K \simeq \Omega \cyy^K \times \prod_{i=1}^m \Omega X_i.\] \qed
\end{thm}

Combining these results, we can give a new retraction result. 

\begin{lem}\label{lem: retractgeneral}
    Let $\uxa^K$ be a polyhedral product, and let $Y_i$ denote the homotopy fibre of the inclusion $A_i \to X_i$. Then for each minimal missing face $\sigma$ of $K$, the space $\Omega(*_{j \in \sigma} Y_j)$ retracts off $\Omega \uxa^K$. 
\end{lem}

\begin{proof}
    By Lemma \ref{lem: left inverse}, $\uxa^{\partial \sigma}$ retracts off $\uxa^K$. Furthermore, by Theorem \ref{fibreandloop}, $\Omega \cyy^{\partial \sigma}$ retracts off $\Omega \uxa^{\partial \sigma}$. By \cite{porter}, $\cyy^{\partial \sigma} \simeq *_{j \in \sigma} Y_i$, and so by composing these retractions and the homotopy equivalence, we obtain that $\Omega(*_{j \in \sigma} Y_i)$ retracts off $\Omega \uxa^K$.  
\end{proof}

\subsection{Known results on Moore's conjecture for polyhedral products}

Moore's conjecture with respect to polyhedral products has been considered. Theorem \ref{moore1} was shown in \cite{HAO_2018} for $n_i \geq 2$ and in \cite{Kim2018} for $n_i =1$, and the two theorems are combined here.

\begin{thm}\cite[Theorem 1.1]{HAO_2018}\cite[Theorem 1.1]{Kim2018}\label{moore1} Let $K$ be a simplicial complex on the vertex set $[m]$ and let $\uxa$ be any sequence
of pairs $(D^{n_i},S^{n_i-1})$. Suppose either:
\begin{itemize}
    \item $n_i \geq 2$ for every $i \in [m]$, or 
    \item $n_i = 1$ for every $i \in [m]$ and every minimal missing face of $K$ has at least 3 vertices.
\end{itemize} Then:
\begin{itemize}
    \item $\uxa^K$ is rationally elliptic if and only if the minimal missing faces of K are mutually
disjoint;
    \item Moore's conjecture holds for $\uxa^K$. \qed
\end{itemize} 
\end{thm} 

In \cite{HAO_2018}, the authors determined necessary and sufficient conditions to determine when a polyhedral product was rationally elliptic.

\begin{thm}\cite[Theorem 1.2]{HAO_2018}\label{moore}
    Let $K$ be a simplicial complex on the vertex set $[m]$ and let $\uxa$ be any sequence of
pointed, path-connected  pairs. For $1 \leq i \leq m$, let $Y_i$ be the homotopy fibre of the inclusion $A_i \to X_i$
and suppose each $Y_i$ is rationally nontrivial. Then the polyhedral product $\uxa^K$ is rationally elliptic if
and only if three conditions hold:
\begin{enumerate}
    \item each $X_i$ is rationally elliptic;
    \item all the minimal missing faces of $K$ are mutually disjoint;
    \item if $v$ is a vertex of a minimal missing face of $K$ then $Y_v$ is rationally homotopy equivalent to
a sphere.
\end{enumerate} \qed
\end{thm}

Suppose additionally that each $X_i$ and $A_i$ is simply connected. If $Y_i$ is rationally trivial for some $i$, the inclusion $A_i \to X_i$ is a rational equivalence, and so if $Y_i$ is rationally trivial, there is a rational homotopy equivalence $\uxa^K \simeq_{\mathbb{Q}} X_i \times \uxa^{K \setminus i}$. Therefore $\uxa^K$ is rationally elliptic if and only if $X_i$ and $\uxa^{K \setminus i}$ are rationally elliptic. Repeating this process removes all rationally trivial fibres. 

We shall also use the following result. 

\begin{thm}\cite[Theorem 4.2]{HAO_2018}\label{technicalmoore}
    Let $K$ be a simplicial complex on the vertex set $[m]$ and let $\uxa$ be a sequence of pointed path connected $CW$-pairs. The following hold:
    \begin{itemize}
        \item if $MMF(K) = \{ \sigma_1, \dots, \sigma_n\}$ and these minimal missing faces are mutually disjoint, then there is a homotopy equivalence \[\Omega \uxa^K \simeq (\prod_{i=1}^m \Omega X_i) \times (\prod_{j=1}^n \Omega \cyy^{\partial \sigma_j});\]
        \item if $\sigma_1$ and $\sigma_2$ are minimal missing faces of $K$ with nontrivial intersection then the space $\Omega(\cyy^{\partial \sigma_1} \vee\cyy^{\partial \sigma_2})$ retracts off $\Omega \uxa^K$. 
    \end{itemize}
\end{thm}

\section{New results on Moore's conjecture for polyhedral products}

Here we generalise Theorem \ref{moore1}, and determine when $\caa^K$ verifies Moore's conjecture. We first state some well-known results. Let $\mathcal{W}$ denote the collection of topological spaces that are homotopy equivalent to a finite-type wedge of simply-connected spheres, and recall that $\mathcal{P}$ is the collection of $H$-spaces which are homotopy equivalent to a finite-type product of spheres and loops on simply-connected spheres. For proofs of \eqref{lem: PWspaces-suspension} and \eqref{lem: PWspaces-loop} see \cite[Lemma 4.1]{soton467338} and \cite{Hilton}. For \eqref{lem: PWspaces-join-suspension}, note that this follows from the fact $S^n \wedge S^m \cong S^{n+m}$.

\begin{lem}\label{lem: PWspaces}
The classes $\mathcal{P}$ and $\mathcal{W}$ satisfy the following properties.
\begin{enumerate}
    \item\label{lem: PWspaces-suspension}
    If $X \in \mathcal{P}$, then $\Sigma X \in \mathcal{W}$;

    \item\label{lem: PWspaces-loop}
    If $X \in \mathcal{W}$, then $\Omega X \in \mathcal{P}$ ;

    \item\label{lem: PWspaces-join-suspension}
    If $X, Y \in \mathcal{W}$, then $X \wedge Y \in \mathcal{W}$.
\end{enumerate} \qed
\end{lem}

\begin{thm}\label{thm: aiinw}
     Let $K$ be a simplicial complex on the vertex set $[m]$ and let $\caa$ be a sequence of pairs such that for all $i \in [m]$, $A_i$ is a finite $CW$-complex, $\Sigma A_i \in \mathcal{W}$, and $A_i$ is rationally nontrivial. Then 
     \begin{enumerate}
         \item $\caa^K$ is rationally elliptic if and only if $K$ has mutually disjoint minimal missing faces and if the vertex $i$ lies in a minimal missing face of $K$, then $A_i$ must be rationally homotopy equivalent to a sphere;
         \item  Moore's conjecture holds for $\caa^K$.  
     \end{enumerate}
\end{thm}
\begin{proof}
\begin{enumerate}
    \item  Theorem \ref{moore} gives necessary and sufficient conditions for $\caa^K$ to be rationally elliptic. Condition $1$ of Theorem \ref{moore} is always satisfied for the pair $\caa$, as $CA_i$ is trivially rationally elliptic. Condition $2$ implies that $K$ must have mutually disjoint minimal missing faces. The homotopy fibre $Y_i$ is $A_i$ for the pair $\caa$, and so we obtain that $A_i \simeq_{\mathbb{Q}} S^n$ for every vertex $i$ in a minimal missing face of $K$. 
    
   \item If $K$ has mutually disjoint minimal missing faces, there is a simplicial isomorphism $K ~\cong K_0*K_1* \dots * K_n$, where $K_0$ is a simplex and for $1 \leq j \leq n$, $K_j= \partial \sigma_j$, where $\sigma_j = \{1_j, \dots, n_j\}$, so $|\sigma_j| = n_j$. (This result is generally known, but a proof may be found in \cite{MR3633135}). If there are no missing faces of $K$, then $K = \Delta^{m-1}$, and so $\caa^K \simeq \prod_{i=1}^m CA_i \simeq *$, as each $CA_i$ is contractible. By the definition of a polyhedral product, there is a homeomorphism $\uxa^{K_1*K_2} \cong \uxa^{K_1} \times \uxa^{K_2} $. Therefore there is a homotopy equivalence
    \[\Omega \caa^{K} \simeq \prod_{j=1}^n\Omega \caa^{\partial \sigma_j}.\] Consider $\Omega \caa^{\partial \sigma_j}$. By \cite[Proposition 2.3]{grbic2011homotopy}, $\caa^{\partial \sigma_j} \simeq \Sigma^{n_j-1} A_{i_j} \wedge \dots \wedge A_{n_j}$. As each $\Sigma A_{i_j} \in \mathcal{W}$, $\caa^{\partial \sigma_j} \in \mathcal{W}$ by Lemma \ref{lem: PWspaces} \eqref{lem: PWspaces-join-suspension}. By Lemma \ref{lem: PWspaces} \eqref{lem: PWspaces-loop}, $\Omega \caa^{\partial \sigma_j} \in \mathcal{P}$ for all $j \in \{1,\dots,n\}$. Therefore $\Omega \caa^{K} \in \mathcal{P}$, and by Lemma \ref{lem: PWspaces-moore}, Moore's conjecture holds. 

    Now suppose $K$ has at least two minimal missing faces with a non-trivial intersection. Theorem \ref{moore} implies $\caa^K$ is not rationally elliptic, and therefore must be rationally hyperbolic. Moreover, Theorem \ref{technicalmoore} gives that $\Omega (\caa^{\partial \sigma_1} \vee \caa^{\partial \sigma_2})$ retracts off $\Omega \caa^K$. By Lemma \ref{lem: PWspaces} \eqref{lem: PWspaces-join-suspension}, $\caa^{\partial \sigma_1}$ and $\caa^{\partial \sigma_2}$ are elements of $\mathcal{W}$, and so $\caa^{\partial \sigma_1} \vee \caa^{\partial \sigma_2} \in \mathcal{W}$. By assumption, each $A_i$ is rationally nontrivial, and therefore $\caa^{\partial \sigma_1} \vee \caa^{\partial \sigma_2}$ must be a wedge of at least two simply-connected spheres. By \cite{NeisendorferSelick1982}, a wedge of at least two simply-connected spheres has no homotopy exponent at any prime $p$. Therefore Moore's conjecture holds in this case. 
    \end{enumerate}
\end{proof}

\begin{thm}\label{thm: mooretorsionfreepp}
   Let $K$ be a simplicial complex on the vertex set $[m]$ and let $\caa$ be a sequence of pairs such that for all $i \in [m]$, $A_i$ is a finite path-connected $CW$-complex, and each $A_i$ has torsion-free homology. Then if $\caa^K$ is rationally hyperbolic, $\caa^K$ has no homotopy exponent for any odd prime $p$.
\end{thm}
\begin{proof}
    We first reduce to the case where each $A_i$ is rationally nontrivial. Consider $\caa^K$ where for some $i$, $A_i$ is rationally contractible. Then $\caa^K \simeq_{\mathbb{Q}} \caa^{K\setminus i}$, and so vertices associated with rationally trivial $A_i$'s can be deleted until only rationally non trivial ones remain without changing the rational homotopy type. Denote this simplicial complex $K'$. Note that $K'$ is a full subcomplex of $K$, and so by Lemma \ref{lem: left inverse}, $\caa^{K'}$ retracts off $\caa^K$. Therefore, if $\caa^{K'}$ has no homotopy exponent, neither does $\caa^K$. Since $\caa^{K'} \simeq_{\mathbb{Q}} \caa^K$, the polyhedral product $\caa^{K'}$ is also rationally hyperbolic. Replacing $K$ by $K'$ and relabelling the vertices, we may therefore assume from now on that every $A_i$ is rationally non trivial. 
    
    Assume $\caa^K$ is rationally hyperbolic. There are two cases to consider; either $K$ has mutually disjoint minimal missing faces, or $K$ has at least two minimal missing faces with a non-trivial intersection. If $K$ has mutually disjoint missing faces, and as above, we obtain a homotopy equivalence  \[\Omega \caa^{K} \simeq \prod_{j=1}^n\Omega \caa^{\partial \sigma_j}. \] As $\caa^K$ is rationally hyperbolic, at least one factor in this product must be rationally hyperbolic. Denote this factor by $\caa^{\partial \sigma_k}$, where $\sigma_k =\{1_k,\dots ,n_k\}$. By \cite[Proposition 2.3]{grbic2011homotopy}, $\caa^{\partial \sigma_k}$ is homotopy equivalent to $\Sigma^{n_k-1} A_{1_k} \wedge~\dots \wedge~ A_{n_k}$. As each $A_{i_k}$ has torsion-free homology, the reduced K{\"u}nneth formula implies $\Sigma^{n_k-1} A_{1_k} \wedge~\dots \wedge~ A_{n_k}$ also has torsion-free homology. Therefore Theorem \ref{thm: mooretorsionfree} implies that $\caa^{\partial \sigma_k}$ has no homotopy exponent at any odd prime $p$. As $\Omega \caa^{\partial \sigma_k}$ retracts off $\Omega \caa^K$, $\Omega \caa^K$ has no homotopy exponents at any odd prime $p$. 
   
    Now suppose $K$ has at least two minimal missing faces with a nontrivial intersection. As before, $\Omega (\caa^{\partial \sigma_i} \vee~ \caa^{\partial \sigma_j})$ retracts off $\Omega \caa^K$. Both $\caa^{\partial \sigma_i} \simeq \Sigma^{n_i-1} A_{1_i} \wedge\dots \wedge A_{n_i}$ and $\caa^{\partial \sigma_j} \simeq \Sigma^{n_j-1} A_{1_j} \wedge\dots \wedge A_{n_j}$ are finite suspensions, and so we can rewrite this as $\Sigma ((\Sigma^{n_i-2} A_{1_i} \wedge\dots \wedge A_{n_i}) \vee ( \Sigma^{n_j-2} A_{1_j} \wedge\dots \wedge A_{n_j})$, which is a finite suspension. By \cite[Theorem 24.5]{FelixHalperinThomas2001} both $\caa^{\partial \sigma_i}$ and $\caa^{\partial \sigma_j}$ are rationally homotopy equivalent to a finite wedge of spheres, so $\caa^{\partial \sigma_i} \vee \caa^{\partial \sigma_j}$ is also rationally homotopy equivalent to a finite wedge of spheres. Each $A_i$ is rationally nontrivial, and so there are at least two spheres present in such a wedge. Therefore $\caa^{\partial \sigma_i} \vee \caa^{\partial \sigma_j}$ is rationally hyperbolic. Furthermore, as each $A_i$ has torsion-free homology, $\caa^{\partial \sigma_i} \vee \caa^{\partial \sigma_j}$ has torsion-free homology, and by Theorem \ref{thm: mooretorsionfree}, $\caa^{\partial \sigma_i} \vee \caa^{\partial \sigma_j}$ has no homotopy exponents at all odd primes $p$. As $\Omega (\caa^{\partial \sigma_i} \vee \caa^{\partial \sigma_j})$ retracts off $\Omega \caa^K$, $\caa^K$ has no homotopy exponent at any odd prime.
\end{proof}

\section{Results on Moore's conjecture which hold for almost all primes}

In this section, we consider Moore's conjecture for polyhedral products of the form $\ux^K$ at all but finitely many primes. For every vertex $i \in [m]$, $X_i$ retracts off $\ux^K$. Therefore, if there exists some $i$ such that $X_i$ is rationally hyperbolic, $\ux^K$ is rationally hyperbolic. Furthermore, if at a fixed prime $p$, $X_i$ has no homotopy exponent, neither does $\ux^K$. Therefore, we consider the case where the spaces $X_i$ are rationally elliptic. We first establish that the loop space of a rationally non-trivial simply connected finite $CW$-complex has infinite-dimensional rational homology. When working $p$-locally, we use $\mathcal{P}$ and $\mathcal{W}$ for the corresponding classes formed from $p$-local spheres; the properties outlined in Lemma \ref{lem: PWspaces} hold in this case as well. The following result follows from the Serre spectral sequence. 

\begin{lem}\label{lem: hyperbolicloopspectral}
    Let $X$ be a simply connected finite $CW$-complex. If $X$ is rationally non-trivial, then $\mathrm{dim} \Tilde{H}_*(\Omega X ;\mathbb{Q}) = \infty$. \qed
\end{lem}

We now give the main theorem of this section.

\begin{thm}\label{thm: uxmoores}
    Let $\ux^{K}$ be a polyhedral product. Suppose that there is a minimal missing face $\sigma$ of $K$ such that, for each $j \in \sigma$, the space $X_j$ is a simply connected finite $CW$-complex which is rationally elliptic and rationally non-trivial. After localising at all but finitely many primes, there exists a wedge $W$ of infinitely many simply connected $p$-local spheres such that $\Omega W$ retracts off
\[
\Omega\bigl((\ux^K)_{(p)}\bigr).
\]
Consequently, $\ux^K$ is mod-$p^r$-hyperbolic for every $r\geq 1$, is rationally hyperbolic, and has no homotopy exponent at $p$. In particular, Moore's conjecture holds for $\ux^K$ at $p$.
\end{thm}
\begin{proof}
    By Lemma \ref{lem: retractgeneral}, $\Omega(*_{j \in \sigma} \Omega X_j)$ retracts off $\Omega \ux^K$. After localising at all but finitely many primes, by Theorem \ref{thm : loop decomp}, $\Omega (X_j)_{(p)} \in \mathcal{P}$ for each $j \in \sigma$. Recall that $A*B \simeq \Sigma A \wedge B$. Therefore, by Lemma \ref{lem: PWspaces}, $(*_{j \in \sigma} \Omega (X_j)_{(p)}) \in \mathcal{W}$, so $(*_{j \in \sigma} \Omega X_j)_{(p)}$ is a wedge of simply connected $p$-local spheres. Denote this wedge as $W$. Then after localising at all but finitely many primes, by Lemma \ref{lem: retractgeneral}, $\Omega W$ retracts off  $\Omega\bigl((\ux^K)_{(p)}\bigr)$. 
    
    By Lemma \ref{lem: hyperbolicloopspectral}, $\mathrm{dim} \Tilde{H}_*(W; \mathbb{Q}) = \mathrm{dim} \Tilde{H}_*((*_{j \in \sigma} \Omega X_j) ; \mathbb{Q}) = \prod_{j \in \sigma} \mathrm{dim} \Tilde{H}_*(\Omega X_j ; \mathbb{Q}) = \infty$. As $W$ is a wedge of $p$-local spheres, there are infinitely many wedge summands appearing in $W$. Therefore, after localising at all but finitely many primes, a wedge of at least 2 $p$-local spheres, denoted $S^a_{(p)} \vee S^b_{(p)}$, that retracts off $W$. As $\Omega W$ is a retract of $\Omega (\ux^K_{(p)})$, it follows $\Omega(S^a_{(p)} \vee S^b_{(p)})$ retracts off $\Omega(\ux^K_{(p)})$. Therefore, $\ux^K$ is rationally hyperbolic, and for all but finitely many primes, is mod-$p^r$ hyperbolic  for every $r \geq 1$ and has no homotopy exponent at $p$, by Examples \ref{2.2}, \ref{2.4} and \ref{2.6} respectively. In particular, Moore's conjecture holds for $\ux^K$ at $p$.
\end{proof}

\begin{corollary}
    Let $X_1,\dots,X_m$ be simply connected finite $CW$-complexes which are rationally elliptic and rationally non-trivial. Then $\ux^K$ satisfies Moore's conjecture for all but finitely many primes. 
\end{corollary}
\begin{proof}
    If $K$ is not a simplex, the result follows from Theorem \ref{thm: uxmoores}. Suppose $K = \Delta^{m-1}$. Then   $\ux^K = \prod_{i=1}^m X_i$, and the product is rationally elliptic. Theorem \ref{thm : loop decomp} implies that this product has finite $p$-exponent for all but finitely many primes. 
\end{proof}

\begin{corollary}\label{5.4}
    Let $K$ be a simplicial complex on $[m]$, and suppose that each $X_i$ is a nonempty finite product of simply connected spheres. Then:
    \begin{itemize}
        \item If $K$ is a simplex, $\ux^K$ is rationally elliptic and has a finite homotopy exponent at every prime,
        \item If $K$ is not a simplex, $\ux^K$ is rationally hyperbolic, mod-$p^r$ hyperbolic for every prime $p$ and every $r \geq 1$, and has no homotopy exponent at any prime. 
    \end{itemize}
    Therefore, Moore's conjecture holds for $\ux^K$. 
\end{corollary}
\begin{proof}
    If $K$ is a simplex, then $\ux^K = \prod_{i=1}^m X_i$, and so is a finite product of simply connected spheres. It is therefore rationally elliptic and has finite homotopy exponent at every prime. 

    Now suppose $K$ is not a simplex. Then for $\sigma \in MMF(K)$, by Lemma \ref{lem: retractgeneral}, $\Omega (*_{j \in \sigma} \Omega X_j)$ retracts off $\Omega \ux^K$. Furthermore, as $\Omega X_j \in \mathcal{P}$ integrally, by Lemma \ref{lem: PWspaces} $*_{j \in \sigma} \Omega X_j \in \mathcal{W}$ before localisation. By the same argument as Theorem \ref{thm: uxmoores}, the result follows. 
\end{proof}

\bibliographystyle{alpha}
\bibliography{bib}

@article {grbic2011homotopy,
    AUTHOR = {Grbi\'{c}, J. and Theriault, S.},
     TITLE = {The homotopy type of the polyhedral product for shifted
              complexes},
   JOURNAL = {Adv. Math.},
    VOLUME = {245},
      YEAR = {2013},
     PAGES = {690--715},
      ISSN = {0001-8708,1090-2082},
   MRCLASS = {55P15 (13F55 52B70 52C35)},
  MRNUMBER = {3084441},
MRREVIEWER = {Carles\ Broto},
       DOI = {10.1016/j.aim.2013.05.002},
       URL = {https://doi.org/10.1016/j.aim.2013.05.002},
}

@incollection{soton467338,
           month = {March},
           title = {Polyhedral products for wheel graphs and their generalizations},
          author = {S. Theriault},
            BOOKTITLE = {Toric {T}opology and {P}olyhedral {P}roducts},
    SERIES = {Fields Inst. Commun.},
    VOLUME = {89},
     PAGES = {295--311},
 PUBLISHER = {Springer, Cham},
      YEAR = {2024},
      ISBN = {9783031572036; 9783031572043},
   MRCLASS = {99-06},
  MRNUMBER = {4789603},
       DOI = {10.1007/978-3-031-57204-3\_15},
       URL = {https://doi.org/10.1007/978-3-031-57204-3_15},
}

@article {HAO_2018,
    AUTHOR = {Hao, Y. and Sun, Q. and Theriault, S.},
     TITLE = {Moore's conjecture for polyhedral products},
   JOURNAL = {Math. Proc. Cambridge Philos. Soc.},
    VOLUME = {167},
      YEAR = {2019},
    NUMBER = {1},
     PAGES = {23--33},
      ISSN = {0305-0041,1469-8064},
   MRCLASS = {55Q52 (55P62)},
  MRNUMBER = {3968818},
MRREVIEWER = {Daisuke\ Kishimoto},
       DOI = {10.1017/s0305004118000154},
       URL = {https://doi.org/10.1017/s0305004118000154},
}

@article {porter,
    AUTHOR = {Porter, G. J.},
     TITLE = {The homotopy groups of wedges of suspensions},
   JOURNAL = {Amer. J. Math.},
    VOLUME = {88},
      YEAR = {1966},
     PAGES = {655--663},
      ISSN = {0002-9327,1080-6377},
   MRCLASS = {55.45},
  MRNUMBER = {200926},
MRREVIEWER = {N.\ Kuiper},
       DOI = {10.2307/2373148},
       URL = {https://doi.org/10.2307/2373148},
}

@article{Denham_2007,
   title={Moment-angle Complexes, Monomial Ideals and Massey Products},
   volume={3},
   ISSN={1558-8602},
   url={http://dx.doi.org/10.4310/PAMQ.2007.v3.n1.a2},
   DOI={10.4310/pamq.2007.v3.n1.a2},
   number={1},
   JOURNAL = {Pure Appl. Math. Q.},
   publisher={International Press of Boston},
   author={Denham, G. and Suciu, A. I.},
   year={2007},
   pages={25–60} }

@article {MR3633135,
    AUTHOR = {Bahri, A. and Bendersky, M. and Cohen, F. R. and Gitler, S.},
     TITLE = {On the free loop spaces of a toric space},
   JOURNAL = {Bol. Soc. Mat. Mex. (3)},
    VOLUME = {23},
      YEAR = {2017},
    NUMBER = {1},
     PAGES = {257--265},
      ISSN = {1405-213X,2296-4495},
   MRCLASS = {55P62 (55P35 55U10)},
  MRNUMBER = {3633135},
MRREVIEWER = {Jelena\ Grbi\'{c}},
       DOI = {10.1007/s40590-016-0124-8},
       URL = {https://doi.org/10.1007/s40590-016-0124-8},
}

@incollection {bahri2024polyhedralproductsfeatureshomotopy,
    AUTHOR = {Bahri, A. and Bendersky, M. and Cohen, F. R.},
     TITLE = {Polyhedral products and features of their homotopy theory},
 BOOKTITLE = {Handbook of homotopy theory},
    SERIES = {CRC Press/Chapman Hall Handb. Math. Ser.},
     PAGES = {103--144},
 PUBLISHER = {CRC Press, Boca Raton, FL},
      YEAR = {[2020] \copyright 2020},
      ISBN = {978-0-815-36970-7},
   MRCLASS = {55P42 (13F55 14M25 55Q15)},
  MRNUMBER = {4197983},
MRREVIEWER = {Jes\'us\ Gonz\'alez},
}

@article{Hilton,
    author = {Hilton, P. J.},
    title = {On the Homotopy Groups of the Union of Spheres},
    journal = {J. Lond. Math. Soc.},
    volume = {s1-30},
    number = {2},
    pages = {154-172},
    year = {1955},
    month = {04},
    issn = {0024-6107},
    doi = {10.1112/jlms/s1-30.2.154},
    url = {https://doi.org/10.1112/jlms/s1-30.2.154},
    eprint = {https://academic.oup.com/jlms/article-pdf/s1-30/2/154/2474879/s1-30-2-154.pdf},
}

@article{McW,
 ISSN = {00029939, 10886826},
 URL = {http://www.jstor.org/stable/2046328},
 author = {C. A. McGibbon and C. W. Wilkerson},
 journal = {Proc. Amer. Math. Soc.},
 number = {4},
 pages = {698--702},
 publisher = {American Mathematical Society},
 title = {Loop Spaces of Finite Complexes at Large Primes},
 urldate = {2025-05-03},
 volume = {96},
 year = {1986}
}

@article{moore1,
 AUTHOR = {Cohen, F. R. and Moore, J. C. and Neisendorfer, J. A.},
     TITLE = {Torsion in homotopy groups},
   JOURNAL = {Ann. of Math. (2)},
  FJOURNAL = {Annals of Mathematics. Second Series},
    VOLUME = {109},
      YEAR = {1979},
    NUMBER = {1},
     PAGES = {121--168},
      ISSN = {0003-486X},
   MRCLASS = {55Q40},
  MRNUMBER = {519355},
MRREVIEWER = {V.\ P.\ Snaith},
       DOI = {10.2307/1971269},
       URL = {https://doi.org/10.2307/1971269},
}

@article{Moore2,
    AUTHOR = {Cohen, F. R. and Moore, J. C. and Neisendorfer, J. A.},
     TITLE = {The double suspension and exponents of the homotopy groups of
              spheres},
   JOURNAL = {Ann. of Math. (2)},
  FJOURNAL = {Annals of Mathematics. Second Series},
    VOLUME = {110},
      YEAR = {1979},
    NUMBER = {3},
     PAGES = {549--565},
      ISSN = {0003-486X},
   MRCLASS = {55Q40},
  MRNUMBER = {554384},
MRREVIEWER = {M.\ Mahowald},
       DOI = {10.2307/1971238},
       URL = {https://doi.org/10.2307/1971238},
}

@inbook{Neisendorfer,
url = {https://doi.org/10.1515/9781400882113-003},
title = {II. The Exponent of a Moore Space},
booktitle = {Algebraic Topology and Algebraic K-Theory (AM-113), Volume 113},
author = {J. A. Neisendorfer},
publisher = {Princeton University Press},
address = {Princeton},
pages = {35--71},
doi = {doi:10.1515/9781400882113-003},
isbn = {9781400882113},
year = {1988},
lastchecked = {2025-05-05}
}

@article {Serre,
    AUTHOR = {Serre, J.P.},
     TITLE = {Homologie singuli\`ere des espaces fibr\'es. {A}pplications},
   JOURNAL = {Ann. of Math. (2)},
  FJOURNAL = {Annals of Mathematics. Second Series},
    VOLUME = {54},
      YEAR = {1951},
     PAGES = {425--505},
      ISSN = {0003-486X},
   MRCLASS = {56.0X},
  MRNUMBER = {45386},
MRREVIEWER = {W.\ S.\ Massey},
       DOI = {10.2307/1969485},
       URL = {https://doi.org/10.2307/1969485},
}

@article{HuangTheriault2024,
  author    = {Huang, R. and Theriault, S.},
  title     = {Exponential growth in the rational homology of free loop spaces and in torsion homotopy groups},
  journal   = {Ann. Inst. Fourier (Grenoble)},
  volume    = {74},
  number    = {4},
  year      = {2024},
  pages     = {1365--1382},
  doi       = {10.5802/aif.3627}
}

@article{Boyde2022,
  author    = {Boyde, Guy},
  title     = {p-Hyperbolicity of homotopy groups via K-theory},
  journal   = {Math. Z.},
  volume    = {301},
  number    = {1},
  pages     = {977--1009},
  year      = {2022},
  doi       = {10.1007/s00209-021-02917-1},
  url       = {https://doi.org/10.1007/s00209-021-02917-1},
  issn      = {1432-1823}
}

@inproceedings{NeisendorferSelick1982,
  author    = {Neisendorfer, J. A. and Selick, P.},
  title     = {Some examples of spaces with or without exponents},
  booktitle = {Current Trends in Algebraic Topology, Part 1},
  address   = {London, Ont., Canada},
  year      = {1981},
  pages     = {343--357},
  series    = {CMS Conference Proceedings},
  volume    = {2},
  publisher = {American Mathematical Society},
  location  = {Providence, RI},
}

@article{Kim2018,
  author    = {Kim, J. H.},
  title     = {Real polyhedral products, Moore's conjecture, and simplicial actions on real toric spaces},
  journal   = {Bull. Korean Math. Soc.},
  volume    = {55},
  number    = {4},
  pages     = {1051--1063},
  year      = {2018},
  doi       = {10.4134/BKMS.b170523},
  url       = {http://bkms.kms.or.kr/journal/view.html?doi=10.4134/BKMS.b170523},
  issn      = {1015-8634, 2234-3016},
  publisher = {The Korean Mathematical Society},
}

@misc{stanton2026anicksconjecturepolyhedralproducts,
      title={Anick's conjecture for polyhedral products}, 
      author={L. Stanton and F. Vylegzhanin},
      year={2026},
      eprint={2506.15573},
      archivePrefix={arXiv},
      primaryClass={math.AT},
      url={https://arxiv.org/abs/2506.15573}, 
      note={arXiv: 2506.15573}
}

@article{Selick_1983, title={On conjectures of Moore and Serre in the case of torsion-free suspensions}, volume={94}, DOI={10.1017/S0305004100060916}, number={1}, journal={Math. Proc. Cambridge Philos. Soc.}, author={Selick, P.}, year={1983}, pages={53–60}}

@book{FelixHalperinThomas2001,
  author    = {Félix, Y. and Halperin, S. and Thomas, J.C.},
  title     = {Rational Homotopy Theory},
  series    = {Graduate Texts in Mathematics},
  volume    = {205},
  edition   = {1},
  publisher = {Springer},
  address   = {New York, NY},
  year      = {2001},
  doi       = {10.1007/978-1-4613-0105-9},
  isbn      = {978-0-387-95068-6, 978-1-4612-6516-0, 978-1-4613-0105-9},
  pages     = {XXXIII+539},
  series_issn = {0072-5285},
  series_eissn = {2197-5612}
}

@article{Jamessphere,
 ISSN = {0003486X, 19398980},
 URL = {http://www.jstor.org/stable/1970011},
 author = {I. M. James},
 journal = {Ann. of Math. (2)},
 number = {3},
 pages = {407--429},
 publisher = {[Annals of Mathematics, Trustees of Princeton University on Behalf of the Annals of Mathematics, Mathematics Department, Princeton University]},
 title = {The Suspension Triad of a Sphere},
 urldate = {2026-05-01},
 volume = {63},
 year = {1956}
}

@article{Toda1956,
  author    = {Toda, H.},
  title     = {On the double suspension $E_2$},
  journal   = {J. Inst. Polytech. Osaka City Univ. Ser. A},
  volume    = {7},
  year      = {1956},
  pages     = {103--145}
}

@phdthesis{Long1978,
  author       = {Long, J.},
  title        = {Thesis},
  school       = {Princeton University},
  year         = {1978},
  type         = {Ph.D. Thesis}
}

@article{ChacholskiPitschSchererStanley2008,
  author    = {Chachólski, W. and Pitsch, W. and Scherer, J. and Stanley, D.},
  title     = {Homotopy exponents for large H-spaces},
  journal = {Int. Math. Res. Not. IMRN},
    volume = {2008},
    year = {2008},
  pages     = {Art. ID rnn061, 5 pp.},
  doi       = {10.1093/imrn/rnn061}
}

@article{DAVIS198939,
title = {The image of the stable J-homomorphism},
journal = {Topology},
volume = {28},
number = {1},
pages = {39-58},
year = {1989},
issn = {0040-9383},
doi = {https://doi.org/10.1016/0040-9383(89)90031-1},
url = {https://www.sciencedirect.com/science/article/pii/0040938389900311},
author = {D. M. Davis and M. Mahowald}
}

@article {MR721456,
    AUTHOR = {Selick, Paul},
     TITLE = {{$2$}-primary exponents for the homotopy groups of spheres},
   JOURNAL = {Topology},
  FJOURNAL = {Topology. An International Journal of Mathematics},
    VOLUME = {23},
      YEAR = {1984},
    NUMBER = {1},
     PAGES = {97--99},
      ISSN = {0040-9383},
   MRCLASS = {55Q40},
  MRNUMBER = {721456},
MRREVIEWER = {M.\ Mahowald},
       DOI = {10.1016/0040-9383(84)90029-6},
       URL = {https://doi.org/10.1016/0040-9383(84)90029-6},
}
\end{document}